\documentclass[a4paper,11pt, onesided]{article}

\usepackage[ansinew]{inputenc}

\usepackage{amsfonts}
\usepackage{amsmath}
\usepackage{amsthm}

\usepackage{mathrsfs}
\usepackage{amssymb}

\newtheorem{theorem}{Theorem}

\newtheorem{corollary}[theorem]{Corollary}

\begin{document}

\title{Rank probabilities for real random $N\times  N\times 2$ tensors}
\author{G{\"o}ran Bergqvist$^{1}$ and Peter J. Forrester$^{2}$ \\
     $^{1}$ Department of Mathematics, Link{\"o}ping University, \\
      SE-581 83 Link{\"o}ping, Sweden \\
 $^{2}$ Department of Mathematics and Statistics, University of Melbourne, \\
      Victoria 3010, Australia \\
      gober@mai.liu.se, p.forrester@ms.unimelb.edu.au}
\date{June 28, 2011}

\maketitle

\begin{abstract}
We prove that the probability $P_N$ for a real random Gaussian $N\times  N\times 2$ 
tensor to be of real rank $N$ is $P_N=(\Gamma((N+1)/2))^N/G(N+1)$, where $\Gamma(x)$, $G(x)$ denote the gamma and Barnes $G$-functions respectively. 
This is a rational number for $N$ odd and a rational number multiplied by $\pi^{N/2}$ for $N$ even.
The probability to be of rank $N+1$ is $1-P_N$.
The proof makes use of recent results on the probability of having 
$k$ real generalized eigenvalues for real random Gaussian $N\times  N$ matrices. 
We also prove that 
$\log P_N= (N^2/4)\log (e/4)+(\log N-1)/12-\zeta '(-1)+{\rm O}(1/N)$ for large $N$, where $\zeta$ is
the Riemann zeta function. 
\end{abstract}

\vskip.5cm
\noindent
Keywords: {\it tensors, multi-way arrays, typical rank, random matrices}

\vskip.2cm
\noindent
AMS classification codes: 15A69, 15B52, 60B20


\section{Introduction}

The (real) rank of a  real $m\times n\times p$ 3-tensor or 3-way array $\cal T$ is 
the well defined minimal possible value of $r$ in an expansion
\begin{equation}\label{eq:CP}
{\cal T}=\sum_{i=1}^r \,  {\bf u}_i\otimes {\bf v}_i\otimes {\bf w}_i \qquad (
 {\bf u}_i\in{\mathbb R^m}, {\bf v}_i\in{\mathbb R^n}, {\bf w}_i\in {\mathbb R^p})
\end{equation}
 where $\otimes$ denotes the tensor (or outer) product \cite{BEL,CBLC,SL,Kr2}.

If the elements of ${\cal T}$ are choosen randomly according to a continuous 
probability distribution, there is in general (for general $m$, $n$ and $p$) no generic rank,
 i.e., a rank which occurs with probability 1.
Ranks which occur with strictly positive probabilities are called typical ranks.
We assume that all elements are independent and from a standard 
normal (Gaussian) distribution (mean 0, variance 1). Until now, the only analytically 
known probabilities for typical ranks were for $2\times 2\times 2$ and $3\times 3\times 2$ 
tensors  \cite{B,FM}. Thus in the $2\times2\times 2$ case the probability that $r=2$ is $\pi/4$ and the probability that $r=3$ is $1-\pi/4$, while in the $3\times 3\times 2$ case the probability of the rank equaling 3 is the same as the probability of it equaling 4 which is $1/2$.
Before these analytic results the first numerical simulations were performed by 
Kruskal in 1989, for $2\times 2\times 2$ tensors \cite{Kr2}, and the approximate
values $0.79$ and $0.21$ obtained for the probability of ranks $r=2$ and $r=3$ respectively.
For $N\times N\times 2$ tensors ten Berge and Kiers \cite{tBK} have shown that the only 
typical ranks are $N$ and $N+1$. From ten Berge  \cite{tB},
it follows  that the probability $P_N$ for an $N\times N\times 2$ 
tensor to be of rank $N$ is equal to the probability that a pair of real random Gaussian  
$N\times N$ matrices $T_1$ and $T_2$ (the two slices of $\cal T$) has $N$ real 
generalized eigenvalues, i.e., the probability that $\det (T_1-\lambda T_2)=0$ has only real solutions 
$\lambda$ \cite{B,tB}. Knowledge about the expected number of real solutions to
$\det (T_1-\lambda T_2)=0$ obtained by Edelman et al.~\cite{Ed1} led to
the analytical results for $N=2$ and $N=3$ in \cite{B}.
Forrester and Mays \cite{FM} have recently determined the probabilities
$p_{N,k}$ that $\det (T_1-\lambda T_2)=0$ has $k$ real solutions, and we here apply
the results to $P_N=p_{N,N}$ to obtain explicit expressions for the probabilities for 
all typical ranks of  $N\times N\times 2$ tensors for arbitrary $N$, hence settling 
this open problem for tensor decompositions. We also determine the precise
asymptotic decay of $P_N$ for large $N$ and give some recursion formulas 
for $P_N$.


\section{Probabilities for typical ranks of  $N\times N\times 2$ tensors}

As above, assume that $T_1$ and $T_2$ are real random Gaussian $N\times N$  matrices 
and let $p_{N,k}$ be the probability that $\det (T_1-\lambda T_2)=0$ has $k$ real solutions.
Then Forrester and Mays  {\cite{FM}} prove:

\begin{theorem}\label{FM}
 Introduce the generating function
\begin{equation}
Z_N(\xi)=\sum_{k=0}^N{}^* \xi^kp_{N,k}
\end{equation}
where the asterisk indicates that the sum is over $k$ values of the same parity as $N$.
For $N$ even we have
\begin{equation}\label{FMeven}
Z_N(\xi)={\frac{(-1)^{N(N-2)/8}\Gamma(\frac{N+1}{2})^{N/2}
\Gamma(\frac{N+2}{2})^{N/2}}{2^{N(N-1)/2}\prod_{j=1}^{N}\Gamma(\frac{j}{2})^2}}
\prod_{l=0}^{\frac{N-2}{2}}(\xi^2\alpha_l+\beta_l),
\end{equation}
while for $N$ odd
\begin{equation}\label{FModd}
Z_N(\xi)={\frac{(-1)^{(N-1)(N-3)/8}\Gamma(\frac{N+1}{2})^{(N+1)/2}
\Gamma(\frac{N+2}{2})^{(N-1)/2}}{2^{N(N-1)/2}\prod_{j=1}^{N}\Gamma(\frac{j}{2})^2}}\ \pi\xi
\prod_{l=0}^{\lceil{\frac{N-1}{4}\rceil-1}}(\xi^2\alpha_l+\beta_l)
\prod_{\lceil{\frac{N-1}{4}\rceil}}^{\frac{N-3}{2}}(\xi^2\alpha_{l+1/2}+\beta_{l+1/2})
\end{equation}
Here
\begin{equation}
\alpha_l=\frac{2\pi}{N-1-4l}\ \frac{\Gamma(\frac{N+1}{2})}{\Gamma(\frac{N+2}{2})}
\end{equation}
and
\begin{equation}
\alpha_{l+1/2}=\frac{2\pi}{N-3-4l}\ \frac{\Gamma(\frac{N+1}{2})}{\Gamma(\frac{N+2}{2})}
\end{equation}
The expressions for $\beta_l$ and $\beta_{l+1/2}$ are given in \cite{FM}, but are not needed here, and 
$\lceil{\cdot}\rceil$ denotes the ceiling function.
\end{theorem}
\vskip.2cm

The method used in \cite{FM} relies on first obtaining the explicit form of the element probability density function for 
\begin{equation}\label{TT}
G=T_1^{-1}T_2.
\end{equation}
A real Schur decomposition is used to introduce $k$ real and $(N-k)/2$ complex eigenvalues, with the imaginary part of the latter required to be positive (the remaining $(N-k)/2$ eigenvalues are the complex conjugate of these), for $k=0,2,\ldots,N$ ($N$ even) and $k=1,3,\ldots,N$
 ($N$ odd). The variables not depending on the eigenvalues can be integrated out to give the eigenvalue probability density function, in the event that there are $k$ real eigenvalues. And integrating this over all allowed values of the real and positive imaginary part complex eigenvalues gives $P_{N,k}$.

From Theorem \ref{FM} we derive our main result:

\begin{theorem}\label{prob}
Let $P_N$ denote the probability that a real $N\times N\times 2$ tensor whose elements
are independent and normally distributed with mean 0 and variance 1 has rank $N$. We have
\begin{equation}\label{PG}
P_N=\frac{(\Gamma((N+1)/2))^N}{G(N+1)},
\end{equation}
where \begin{equation}\label{GG}
G(N+1):=(N-1)!(N-2)!\ldots 1!\hspace{0.5cm}(N\in\mathbb{Z}^+)
\end{equation}
is the Barnes $G$-function and $\Gamma(x)$ denotes the gamma function.
More explicitly $P_2=\pi/4$, and for $N\ge 4$ even
\begin{equation}\label{Peven}
P_N=\frac{\pi^{N/2}(N-1)^{N-1}(N-3)^{N-3}\cdot\ldots\cdot 3^3}
{2^{N^2/2}(N-2)^{2}(N-4)^{4}\cdot\ldots\cdot 2^{N-2}}\ ,
\end{equation}
while for $N$ odd
\begin{equation}\label{Podd}
P_N=\frac{(N-1)^{N-1}(N-3)^{N-3}\cdot\ldots\cdot 2^2}
{2^{N(N-1)/2}(N-2)^{2}(N-4)^{4}\cdot\ldots\cdot 3^{N-3}}\ .
\end{equation}
Hence $P_N$ for N odd is a rational number but for 
$N$ even it is a rational number multiplied by $\pi^{N/2}$.
The probability for rank $N+1$ is $1-P_N$.
\end{theorem}

\begin{proof}
From \cite{B} we know that $P_N=p_{N,N}$. Hence, by Theorem \ref{FM}
\begin{equation}
P_N=p_{N,N}=\frac{1}{N!}\frac{d^N}{d\xi^N}Z_N(\xi)
\end{equation}
Since
\begin{equation}
\frac{1}{N!}\frac{d^N}{d\xi^N}\prod_{l=0}^{\frac{N-2}{2}}(\xi^2\alpha_l+\beta_l)=
\prod_{l=0}^{\frac{N-2}{2}}\alpha_l
\end{equation}
and
\begin{equation}
\frac{1}{N!}\frac{d^N}{d\xi^N}\xi
\prod_{l=0}^{\lceil{\frac{N-1}{4}\rceil-1}}(\xi^2\alpha_l+\beta_l)
\prod_{\lceil{\frac{N-1}{4}\rceil}}^{\frac{N-3}{2}}(\xi^2\alpha_{l+1/2}+\beta_{l+1/2})=
\prod_{l=0}^{\lceil{\frac{N-1}{4}\rceil-1}}\alpha_l
\prod_{\lceil{\frac{N-1}{4}\rceil}}^{\frac{N-3}{2}}\alpha_{l+1/2}
\end{equation}
the values of $\beta_l$ and $\beta_{l+1/2}$ are not needed for the determination of $P_N$.
By (\ref{FMeven})  we immediately find
\begin{equation}
P_N={\frac{(-1)^{N(N-2)/8}\Gamma(\frac{N+1}{2})^{N/2}
\Gamma(\frac{N+2}{2})^{N/2}}{2^{N(N-1)/2}\prod_{j=1}^{N}\Gamma(\frac{j}{2})^2}}
\prod_{l=0}^{\frac{N-2}{2}}\alpha_l
\end{equation}
if $N$ is even. For $N$ odd we use (\ref{FModd}) to get
\begin{equation}
P_N={\frac{(-1)^{(N-1)(N-3)/8}\Gamma(\frac{N+1}{2})^{(N+1)/2}
\Gamma(\frac{N+2}{2})^{(N-1)/2}}{2^{N(N-1)/2}\prod_{j=1}^{N}\Gamma(\frac{j}{2})^2}}\ \pi
\prod_{l=0}^{\lceil{\frac{N-1}{4}\rceil-1}}\alpha_l
\prod_{\lceil{\frac{N-1}{4}\rceil}}^{\frac{N-3}{2}}\alpha_{l+1/2}
\end{equation}
Substituting the expressions for $\alpha_l$ and $\alpha_{l+1/2}$ into these formulas we
obtain, after simplifying, for $N$ even
\begin{equation}\label{15}
P_N={\frac{(-1)^{N(N-2)/8}(2\pi)^{N/2}\Gamma(\frac{N+1}{2})^{N}}
{2^{N(N-1)/2}\prod_{j=1}^{N}\Gamma(\frac{j}{2})^2}}
\prod_{l=0}^{\frac{N-2}{2}}\frac{1}{N-1-4l}\ ,
\end{equation}
and for $N$ odd
\begin{equation}\label{16}
P_N={\frac{(-1)^{(N-1)(N-3)/8}(2\pi)^{(N+1)/2}\Gamma(\frac{N+1}{2})^{N}}
{2^{N(N-1)/2+1}\prod_{j=1}^{N}\Gamma(\frac{j}{2})^2}}
\prod_{l=0}^{\lceil{\frac{N-1}{4}\rceil-1}}\frac{1}{N-1-4l}
\prod_{\lceil{\frac{N-1}{4}\rceil}}^{\frac{N-3}{2}}\frac{1}{N-3-4l}\ .
\end{equation}
Now
\begin{align}\label{F1}
\prod_{j=1}^N \Gamma(j/2)^2&=\frac{\Gamma(1/2)}{\Gamma((N+1)/2)}\prod_{j=1}^N\Gamma(j/2)\Gamma((j+1)/2) \nonumber \\
&=\frac{\Gamma(1/2)}{\Gamma((N+1)/2)} \prod_{j=1}^N 2^{1-j}\sqrt{\pi}\,\Gamma(j) \nonumber \\
&=\frac{\Gamma(1/2)}{\Gamma((N+1)/2)}2^{-N(N-1)/2}\pi^{N/2}G(N+1),
\end{align}
where to obtain the second equality use has been made of the duplication formula for the gamma function, and to obtain the third equality the expression (\ref{GG}) for the Barnes $G$-function has been used. Furthermore, for each $N$ even
\begin{align}
(-1)^{N(N-2)/8}\prod_{l=0}^{(N-2)/2}\frac{1}{N-1-4l}
&=\frac{(-1)^{N(N-2)/8}}{(N-1)(N-5)\ldots (N-1-(2N-4))}\notag\\
&=\frac{1}{(N-1)(N-3)\ldots 3 \cdot 1}\notag\\
&=\frac{\Gamma(1/2)}{2^{N/2}\Gamma((N+1)/2)}, \label{F2}
\end{align}
where to obtain the final equation use is made of the fundamental gamma function recurrence
\begin{equation}\label{xG}
\Gamma(x+1)=x\Gamma(x),
\end{equation}
and for $N$ odd
\begin{align}
&(-1)^{(N-1)(N-3)/8}\prod_{l=0}^{\lceil \frac{N-1}{4}\rceil}\frac{1}{N-1-4l}\prod_{\lceil \frac{N-1}{4}\rceil}^{\frac{N-3}{2}}\frac{1}{N-3-4l}\notag\\
&\hspace{1cm}=(-1)^{(N-1)(N-3)/8} 
\left\{ \begin{tabular}{ll}$\displaystyle  \frac{1}{(N-1)(N-5)\ldots 2}\frac{1}{(-4)(-8)\ldots(-N+3)}$,&
$N=3,7,11,\ldots$\\ $\displaystyle   \frac{1}{(N-1)(N-5)\ldots 4}\frac{1}{(-2)(-6)\ldots(-N+3)}$,& $N=5,9,13,\ldots$ \end{tabular}\right.\notag\\
&\hspace{1cm}=\frac{1}{(N-1)(N-3)\ldots 4\cdot 2}\notag\\
&\hspace{1cm}=\frac{1}{2^{(N-1)/2}\Gamma((N+1)/2)}\label{F3}
\end{align}

Substituting (\ref{F1}) and (\ref{F2}) in (\ref{15}) establishes (\ref{PG}) for $N$ even, while the $N$ odd case of (\ref{PG}) follows by substituting (\ref{F1}) and (\ref{F3}) in (\ref{16}), and the fact that 
\begin{equation}\label{Gp}
\Gamma(1/2)=\sqrt{\pi}. 
\end{equation}
The forms (\ref{Peven}) and (\ref{Podd}) follow from (\ref{PG}) upon use of (\ref{GG}), the recurrence (\ref{xG}) and (for $N$ even) (\ref{Gp}).

\end{proof}


\section{Recursion formulas and asymptotic decay}

By Theorem \ref{prob} it is straightforward to calculate $P_{N+1}/P_N$ from either (\ref{PG}) or
(\ref{Peven}) and (\ref{Podd}), and $P_{N+2}/P_N$ from either (\ref{PG}) or (\ref{Peven}) and (\ref{Podd}).
\begin{corollary}\label{rec}
For general $N$
\begin{equation}
P_{N+1}=P_N\cdot \frac{\Gamma(N/2+1)^{N+1}}{\Gamma((N+1)/2)^N}\frac{1}{\Gamma(N+1)},\hspace{0.5cm}P_{N+2}=P_{N}\cdot \frac{((N+1)/2)^{N+2}\Gamma((N+1)/2)^2}{\Gamma(N+2)\Gamma(N+1)}
\end{equation}
 More explicitly, making use of the double factorial
$$
N!!=\left\{ \begin{tabular}{ll} $N(N-2)\ldots 4\cdot 2$, & $N$ even \\ $N(N-2)\ldots 3\cdot 1$, & $N$ odd,
\end{tabular}\right.
$$ for $N$ even we have the recursion formulas
\begin{equation}\label{receven}
P_{N+1}=P_N\cdot \frac{(N!!)^N}{(2\pi)^{N/2}((N-1)!!)^{N+1}}\ ,\quad
P_{N+2}=P_N\cdot\frac{\pi}{2}\cdot\frac{(N+1)^{N+1}}{2^{2N+1}(N!!)^2}
\end{equation}
and for $N$ odd we have
\begin{equation}\label{recodd}
P_{N+1}=P_N\cdot  \frac{\pi^{(N+1)/2}(N!!)^N}{2^{(3N+1)/2}((N-1)!!)^{N+1}}\ ,\quad
P_{N+2}=P_N\cdot\frac{(N+1)^{N+1}}{2^{2N+1}(N!!)^2}.
\end{equation}
\end{corollary}
We can illustrate the pattern for $P_N$ using Theorem \ref{prob} or Corollary  \ref{rec}. One finds
\begin{equation}
P_{2}=\frac{1}{2^{2}}\cdot\pi\ , \quad P_{3}=\frac{1}{2}
\nonumber
\end{equation}
\begin{equation}
P_{4}=\frac{3^3}{2^{10}}\cdot\pi^2\ , \quad P_{5}=\frac{1}{3^{2}}
\nonumber
\end{equation}
\begin{equation}
P_{6}=\frac{5^5\cdot 3^3}{2^{26}}\cdot \pi^3\ , \quad P_{7}=\frac{3^2}{5^{2}\cdot 2^5}
\nonumber
\end{equation}
\begin{equation}
\quad P_{8}=\frac{7^7\cdot 5^5\cdot 3}{2^{48}}\cdot \pi^4 \ ,
\quad P_{9}=\frac{2^4}{7^{2}\cdot 5^4}
\nonumber
\end{equation}
\begin{equation}
\quad  P_{10}=\frac{7^7\cdot 5^5\cdot 3^{17}}{2^{80}}\cdot \pi^5 \ ,
\quad P_{11}=\frac{5^4}{7^{4}\cdot 3^6\cdot 2^5}
\nonumber
\end{equation}
\begin{equation}
\quad P_{12}=\frac{11^{11}\cdot 7^7\cdot 5^5\cdot 3^{15}}{2^{118}}\cdot \pi^6 \ ,
\quad P_{13}=\frac{5^2}{11^{2}\cdot 7^6\cdot 2^4} \ \ \ \ldots
\end{equation}
Numerically, it is clear that $P_N\to 0$ as $N\to \infty$. Some qualitative insight into the rate of decay
 can be obtained by recalling $P_N=p_{N,N}$ and considering the behaviour of $p_{N,k}$ as a function of $k$. Thus we know from \cite{Ed1} that 
 for large $N$, the mean number of real eigenvalues $E_N:=\langle k \rangle_{p_{N,k}}$ is to leading order equal to $\sqrt{\pi N/2}$, and from \cite{FM} that 
 the corresponding variance $\sigma_N^2:=\langle k^2\rangle_{p_{N,k}}-E_N^2$ is to leading order equal to $(2-\sqrt{2})E_N$. The latter reference also shows that  $\lim_{N \to \infty} \sigma_N p_{N,[\sigma_N x+E_N]} = {1 \over \sqrt{2 \pi}} e^{-x^2/2}$, and is thus $p_{N,k}$ is a standard Gaussian distribution  after centering and
 scaling in $k$ by appropriate multiples of $\sqrt{N}$. It follows  that $p_{N,N}$ is, for large $N$, in the large deviation regime of $p_{N,k}$. We remark that this is similarly true of $p_{N,N}$ in the case of eigenvalues of $N\times N$ real random Gaussian matrices (i.e. the individual matrices $T_1,\,T_2$ of (\ref{TT})), for which it is known $p_{N,N}=2^{-N(N-1)/4}$ \cite{Ed1}, \cite[Section 15.10]{Fo10}.

In fact from the exact expression (\ref{PG}) the explicit asymptotic large $N$ form of $P_N$ can readily be calculated. For this, let
\begin{equation}\label{24}
A=e^{-\zeta'(-1)+1/12}= 1.28242712...
\end{equation}
denote the Glaisher-Kinkelin constant, where $\zeta$ is the Riemann zeta function \cite{W}.

\begin{theorem}\label{asy}
For large $N$,
\begin{equation}\label{asfo}
P_N=  N^{1/12}(\frac{e}{4})^{N^2/4}\cdot Ae^{-1/6}(1+{\rm O}(N^{-1}))
\end{equation}
or equivalently
\begin{equation}\label{41'}
\log P_N=(N^2/4)\log(e/4)+(\log N-1)/12-\zeta'(-1)+{\rm O}(1/N). 
\end{equation}
\end{theorem}

\begin{proof}
We require the $x\rightarrow \infty$ asymptotic expansions of the Barnes $G$-function \cite{We00} and the gamma function
\begin{align}
\log G(x+1)&= \frac{x^2}{2}\log x -\frac{3}{4}x^2+\frac{x}{2}\log 2\pi
-\frac{1}{12}\log x +\zeta'(-1)+{\rm O}\Big(\frac{1}{x} \Big), \label{Sg}\\
\Gamma(x+1)& = \sqrt{2\pi x}\Big(\frac{x}{e} \Big)^x\Big( 1+\frac{1}{12x}+{\rm O}\Big(\frac{1}{x^2} \Big)\Big)\label{St}
\end{align}
For future purposes, we note that a corollary of (\ref{St}), and the elementary large $x$ expansion
\begin{equation}\label{ex}
\Big(1+\frac{c}{x} \Big)^x=e^c\Big(1-\frac{c^2}{2x}+{\rm O}\Big( \frac{1}{x^2}\Big) \Big)
\end{equation}
is the asymptotic formula
\begin{equation}\label{sx}
\frac{\Gamma(x+1/2)}{\Gamma(x)}=\sqrt{x}\Big(1-\frac{1}{8x}+{\rm O}\Big(\frac{1}{x^2} \Big) \Big).
\end{equation}

To make use of these expansions, we rewrite (\ref{PG}) as
\begin{equation}\label{Psu}
P_N=\frac{(\Gamma(N/2+1))^N}{G(N+1)}\Big(\frac{\Gamma((N+1)/2)}{\Gamma(N/2+1)} \Big)^N.
\end{equation}
Now, (\ref{sx}) and (\ref{ex}) show that with
\begin{equation}\label{yN}
y:=N/2
\end{equation}
and $y$ large we have
\begin{equation}\label{ye}
\Big( \frac{\Gamma(y+1/2)}{\Gamma(y+1)}\Big)^N = e^{-y \log y}e^{-1/4}\Big( 1+{\rm O}\Big(\frac{1}{y} \Big)\Big).
\end{equation}
Furthermore, in the notation (\ref{yN}) it follows from (\ref{Sg}) and (\ref{St}) and further use of (\ref{ex}) (only the explicit form of the
leading term is now required) that
\begin{equation}\label{yN1}
\frac{\Gamma(N/2+1)^N}{G(N+1)}=e^{-y^2\log(4/e)}e^{y\log y+\frac{1}{12}\log 2y}e^{1/6-\zeta'(-1)}\Big(1+{\rm O}\Big(\frac{1}{y} \Big) \Big).
\end{equation}
Multiplying together (\ref{ye}) and (\ref{yN1}) as required by (\ref{Psu}) and recalling (\ref{yN}) gives (\ref{asfo}). 

Recalling (\ref{24}), the second stated result (\ref{41'}) is then immediate.
\end{proof}

\begin{corollary}\label{asymp}
For large $N$,
\begin{equation}
\frac{P_{N+1}}{P_N}= \Big(\frac{e}{4}\Big)^{(2N+1)/4}(1+{\rm O}(N^{-1}))
\end{equation}
\end{corollary}

This corollary follows trivially from Theorem  \ref{asy}. It can however also be derived directly
from the recursion formulas in Corollary \ref{rec}, without use of Theorem \ref{asy}.

\section*{Acknowledgement}
The work of PJF was supported by the Australian Research Council.


\end{document}